\documentclass[pdflatex,sn-mathphys-num]{sn-jnl}

\usepackage{graphicx}%
\usepackage{multirow}%
\usepackage{amsmath,amssymb,amsfonts}%
\usepackage{amsthm}%
\usepackage{mathrsfs}%
\usepackage{mathbbol}
\usepackage{bbm}
\usepackage[title]{appendix}%
\usepackage{xcolor}%
\usepackage{textcomp}%
\usepackage{manyfoot}%
\usepackage{booktabs}%
\usepackage{algorithm}%
\usepackage{algorithmicx}%
\usepackage{algpseudocode}%
\usepackage{tikz-cd}%
\usepackage{listings}%

\theoremstyle{thmstyleone}%
\newtheorem{theorem}{Theorem}%
\newtheorem{proposition}[theorem]{Proposition}%
\newtheorem{lemma}[theorem]{Lemma}%
\newtheorem{corollary}[theorem]{Corollary}
\newtheorem*{mainthm}{Main Theorem}

\theoremstyle{thmstyletwo}%
\newtheorem{remark}{Remark}%

\theoremstyle{thmstylethree}%
\newtheorem{definition}{Definition}%

\begin{document}

\title[Persistent Stiefel–Whitney Classes]{Persistent Stiefel–Whitney Classes of Tangent Bundles}
\author*[1]{\fnm{Dongwoo} \sur{Gang}}\email{dongwoo.gang@snu.ac.kr}

\affil*[1]{\orgdiv{Department of Mathematical Sciences}, \orgname{Seoul National University}, \orgaddress{\city{Seoul}, \country{South Korea}}}

\abstract{Stiefel--Whitney classes are invariants of the tangent bundle of a smooth manifold, represented as cohomology classes of the base manifold. These classes are essential in obstruction theory, embedding problems, and cobordism theory. In this work, we propose a concrete algorithm to compute persistent cohomology classes representing these characteristic classes. Given a point cloud, we construct a Čech or alpha filtration and apply the Wu formula using the cup product and Steenrod squares. The algorithm runs in polynomial time in the number of simplices without estimating tangent spaces or constructing a classifying map. Under suitable hypotheses on the filtration, the resulting classes represent those of a closed smooth manifold even when the filtration contains additional wedge summands on whose reduced cohomology all positive Steenrod squares vanish. For sufficiently large samples drawn independently and uniformly from the manifold, recovery holds with high probability over an interval of scales. We present experiments on triangulated four-manifolds and on point clouds from image patches and a molecular conformation space.}


\maketitle

\section{Introduction}\label{sec1}

Topological Data Analysis (TDA) applies algebraic topology to extract topological features from point clouds embedded in high-dimensional spaces \citep{edelsbrunner2002topological, zomorodian2005computing, carlsson2009topology}. A central tool, persistent (co)homology, tracks connected components, loops, and higher-dimensional cycles across multiple scales in a filtration.

Persistent cohomology can recover the cohomology groups of a sampled manifold, but these groups need not determine its homotopy type. For example, the torus and the Klein bottle have isomorphic cohomology groups over $\mathbb{Z}/2$, although the torus is orientable and the Klein bottle is not. Similarly, $\mathbb{CP}^2\#\overline{\mathbb{CP}}^2$ and $S^2\times S^2$ have isomorphic cohomology groups over $\mathbb{Z}/n$ for every integer $n$. To distinguish spaces more finely, recent studies have constructed persistent invariants using cohomology operations \citep{memoli2024persistent,lupo2022persistence,medina2025operations}.

Sampling noise can produce additional bars that obscure the features of the sampled manifold \citep{adler2014crackle}. Computing persistent cohomology does not automatically identify which bars arise from the manifold and which are caused by noise. Statistical approaches to separating signal from noise include using confidence bands around the diagonal of a persistence diagram \citep{fasy2014confidence} and selecting smoothing parameters such as kernel bandwidths to maximize the amount of statistically significant persistence \citep{chazal2018robust}.

We seek to recover information about the tangent bundle of the sampled manifold even when additional bars are present. Stiefel–Whitney classes are the characteristic classes of real vector bundles with $\mathbb{Z}/2$ coefficients. Those of the tangent bundle provide information about the underlying manifold \citep{milnor1974characteristic, husemoller1966fibre, steenrod1999topology}. A smooth manifold is orientable if and only if the first Stiefel–Whitney class of its tangent bundle vanishes. Higher Stiefel–Whitney classes impose further obstructions, including constraints on embeddings into Euclidean space of a given dimension.

Several methods have been developed for computing Stiefel–Whitney classes in TDA. \citet{tinarrage2022computing} defined and computed \emph{persistent Stiefel–Whitney classes} of line bundles, using a functorial classifying map into a finite projective space, a finite-dimensional model of $\mathbb{RP}^\infty = \mathbf{Gr}_1(\mathbb{R}^\infty)$. Vector bundles of rank $n$ are instead classified by the Grassmannian $\mathbf{Gr}_n(\mathbb{R}^\infty)$, whose finite-dimensional models are considerably harder to triangulate for $n \geq 2$. Low-dimensional cases have been treated, for instance to compute the second Stiefel–Whitney class of a rank-two bundle over the torus via a classifying map into $\mathbf{Gr}_2(\mathbb{R}^4)$ \citep{tinarrage2021simplicial}.

In an alternative approach, \citet{scoccola2023approximate} computed the first and second Stiefel–Whitney classes using obstructions to reducing the structure group to $SO(d)$ and lifting it to $\operatorname{Pin}^{+}(d)$, respectively. Their method applies to arbitrary vector bundles and can be used for tangent bundles by estimating tangent spaces via local PCA \citep{kambhatla1997dimension}.

\citet{ren2025persistent} independently introduced persistent Stiefel–Whitney classes through functorial pullback bundles and applied them to embedding obstructions in configuration spaces of hard spheres. \citet{paluzo2026learning} used autoencoders to learn local coordinate charts and computed the first Stiefel--Whitney class of the tangent bundle from the signs of the Jacobian determinants of the transition maps. \citet{bohlsen2026counting} applied Scoccola and Perea's Euler class algorithm to compute first Chern numbers from sampled wave polarization vectors and predict the number of topological interface modes.

For a closed smooth manifold, the \emph{Wu formula} determines the Stiefel–Whitney classes of its tangent bundle from the cup product and Steenrod squares on its cohomology ring over $\mathbb{Z}/2$ \citep{milnor1974characteristic}. This gives a way to compute these classes without estimating tangent spaces or constructing a classifying map. Because the cup product and Steenrod squares are functorial, they extend to operations on persistent cohomology \citep{aubrey2011persistent}. Applying the Wu formula to a filtration requires a compatible family of Wu classes on its individual spaces. However, since these spaces need not be manifolds, the existence of such a family does not follow from the classical Wu formula.

In this work, we define the \emph{persistent Wu classes} and the \emph{persistent Stiefel–Whitney classes of dimension $n$} of a filtration by the Wu criterion, using only cohomology operations on its persistent cohomology. We prove that these classes represent those of the tangent bundle when the filtration recovers a closed smooth manifold. Under the hypotheses of the following theorem, this remains true even when some spaces in the filtration are not homotopy equivalent to the sampled manifold.

\begin{mainthm}[Theorem~\ref{thm:wedge}]
    Let $\mathbb{X} = \{X^r\}_{r \in [s,t]}$ be a filtration of topological spaces with $X^r \simeq M \vee Y^r$ compatibly with the inclusions, where $M$ is a closed smooth $n$-dimensional manifold. Assume that all positive Steenrod squares vanish on the reduced cohomology over $\mathbb{Z}/2$ of each noise summand $Y^r$. If $X^t$ is homotopy equivalent to $M$, then for $1\leq k\leq n$ the $k$-th persistent Stiefel–Whitney class of dimension $n$ exists uniquely and represents the $k$-th Stiefel–Whitney class of $TM$.
\end{mainthm}

In the sparse regime of random Čech and Vietoris–Rips complexes, components homotopy equivalent to spheres account for the leading contribution to the expected Betti numbers in positive degrees \citep{kahle2011random, bobrowski2018topology}. For samples affected by weak noise near a manifold, these results suggest modelling the noise summands $Y^r$ by wedges of spheres. All positive Steenrod squares vanish on the reduced cohomology over $\mathbb{Z}/2$ of a sphere and hence of a wedge of spheres, giving the vanishing condition in the theorem.

When the filtration recovers $M$ throughout the interval $[s,t]$ and the inclusions induce isomorphisms on cohomology, Proposition~\ref{swequal} guarantees that the persistent Stiefel–Whitney classes exist and represent those of $TM$. Section~\ref{sec:cech} uses the reconstruction theorem of \citet{niyogi2008finding} to show that this holds with high probability when sufficiently many points are sampled independently and uniformly from $M$.

Computing the Wu classes reduces to solving a linear system over $\mathbb{Z}/2$ by Gaussian elimination. With combinatorial algorithms for the cup product and the Steenrod squares, the whole computation runs in polynomial time in the number of simplices (Section~\ref{sec:algorithm}). The Stiefel–Whitney classes then follow from the Wu formula, and their persistence across the filtration is certified by checking the Wu criterion at the endpoints of the relevant bars.

We illustrate the method on triangulated four-manifolds and on point clouds sampled from image patches and a molecular conformation space. All Python implementations are available at \href{https://github.com/Dongwoo-Gang/sw-using-wu/tree/master}{https://github.com/Dongwoo-Gang/sw-using-wu/tree/master}.

\subsection*{Outline of the Paper}

The paper begins in Section~\ref{sec:background} with a brief introduction to vector bundle theory and persistence theory. In Section~\ref{sec:perssw}, we define the persistent Wu and Stiefel–Whitney classes of dimension $n$ and prove that they represent the Stiefel–Whitney classes of the tangent bundle when the filtration recovers a manifold, including in the presence of noise on which the Steenrod squares vanish. In Section~\ref{sec:cech}, we show that these classes exist with high probability for a sufficiently dense sample. Section~\ref{sec:algorithm} develops the algorithm and analyzes its complexity. In Section~\ref{sec:examples}, we present the computational examples.

\section{Theoretical background}\label{sec:background}

\subsection{Cohomology operations}
In this subsection, we present the combinatorial definitions of two cohomology operations, the cup product and Steenrod squares, following \citet{aubrey2011persistent}. For their construction and topological significance, see \citet{hatcher2005algebraic}, Chapters 3–4. 

Let $X$ be a simplicial complex whose vertices are totally ordered. For a commutative ring \( R \), let \( C^p(X; R) \) and \( H^p(X; R) \) denote the simplicial \( p \)-cochains and the \( p \)-th simplicial cohomology of \( X \) with coefficients in \( R \). The \emph{cup product} of cochains \( \alpha \in C^p(X; R) \) and \( \beta \in C^q(X; R) \) is the cochain $\alpha \smile \beta \in C^{p+q}(X;R)$ defined on each $(p+q)$-simplex $\sigma=[v_0,\ldots,v_{p+q}]$ of $X$ by
\[
(\alpha \smile \beta)(\sigma) = \alpha([v_0, \ldots, v_p]) \cdot \beta([v_p, \ldots, v_{p+q}]).
\]
This operation descends to cohomology, where it is bilinear, associative, and graded-commutative, $[\alpha] \smile [\beta] = (-1)^{pq}\, [\beta] \smile [\alpha]$, and endows \( H^*(X; R) \) with the structure of a graded-commutative algebra. It is functorial with respect to pullbacks. For a continuous map $f: X \to Y$ and classes $u \in H^p(Y; R)$ and $v \in H^q(Y; R)$, we have $f^*(u \smile v) = f^*u \smile f^*v$.

A \emph{cohomology operation} is a natural transformation between cohomology functors. The \emph{Steenrod squares} are the unique family of cohomology operations $\operatorname{Sq}^i: H^p(X; \mathbb{Z}/2) \to H^{p+i}(X; \mathbb{Z}/2)$, whose existence and uniqueness are classical (\citealt{hatcher2005algebraic}, Section 4.L). They satisfy the following properties.
\begin{itemize}
    \item Naturality. For a map \( f: X \to Y \) and \( \alpha \in H^*(Y; \mathbb{Z}/2) \), $f^* \operatorname{Sq}^k(\alpha) = \operatorname{Sq}^k(f^*\alpha)$.
    \item The operation $\operatorname{Sq}^0$ is the identity.
    \item If $|x| = i$, then $\operatorname{Sq}^i(x) = x^2 = x \smile x$.
    \item If $|x| < i$, then $\operatorname{Sq}^i(x) = 0$.
    \item Cartan formula. $\operatorname{Sq}^k(x \smile y) = \sum_{i+j=k} \operatorname{Sq}^i(x) \smile \operatorname{Sq}^j(y)$.
\end{itemize}

In this work, we adopt an approach for computing Steenrod squares based on the cup-$i$ coproduct, as described in \citet{medina2023new}, and implement it in Algorithms~\ref{alg:cupico} and~\ref{alg:steenrod}.

\begin{definition}[\citealt{medina2023new}, Cup-$i$ Coproduct]\label{cupicoprod}
For a simplicial complex \( X \) and \( i \in \mathbb{N} \cup \{0\} \), let \( \sigma=[0,1,\ldots,p] \in X \). Define the cup-$i$ coproduct as
\[
\Delta_i : C_p(X;\mathbb{Z}/2) \rightarrow \big(C_*(X;\mathbb{Z}/2) \otimes C_*(X;\mathbb{Z}/2)\big)_{p+i},
\]
\[
\Delta_i(\sigma) = \sum_{U \subset \{0,1,\ldots,p\},\, |U|=p-i} d_{U^0}(\sigma) \otimes d_{U^1}(\sigma) \in (C_*(X;\mathbb{Z}/2) \otimes C_*(X;\mathbb{Z}/2))_{p+i},
\]
where each subset is written $U=\{u_1 < \cdots < u_{p-i}\}$, and
\[
U^0=\{u_j\in U\mid u_j+j\equiv 0\pmod 2\},\qquad
U^1=\{u_j\in U\mid u_j+j\equiv 1\pmod 2\}.
\]
Here the position $j$ is indexed starting from $1$. For
\(V=\{v_1<\cdots<v_k\}\subseteq\{0,1,\ldots,p\}\), the face
\(d_V(\sigma)\) is obtained by deleting the vertices in the positions listed
in $V$; explicitly,
\(d_V(\sigma)=[0,\ldots,\widehat{v_1},\ldots,\widehat{v_k},\ldots,p]\).
\end{definition}

\begin{theorem}[\citealt{medina2023new}]\label{newsq}
For \( [\alpha] \in H^p(X;\mathbb{Z}/2) \) and \( k \in \mathbb{N} \cup \{0\} \) with $k \leq p$, we have
\[
\operatorname{Sq}^k([\alpha])=[(\alpha \otimes \alpha) \Delta_{p-k}(\cdot)] \in H^{p+k}(X;\mathbb{Z}/2).
\]

\end{theorem}

\subsection{Vector bundle and Stiefel–Whitney class}

In this subsection, we follow \citet{milnor1974characteristic}, Sections 4–8. A \emph{vector bundle} \( \xi = (E, B) \) of rank $n$ consists of a topological space \( E \) called the \emph{total space}, a \emph{base space} \( B \), and a continuous projection \( \pi: E \to B \), such that each fiber \( \pi^{-1}(b) \) is an $n$-dimensional vector space. Locally, for each \( b \in B \), there exists an open neighborhood \( U \subset B \) of $b$ and a homeomorphism $\phi: \pi^{-1}(U) \to U \times \mathbb{R}^n$
such that the restriction of \( \pi \) to \( \pi^{-1}(U) \) corresponds to the projection onto \( U \). A vector bundle is called a \emph{trivial bundle} or \emph{product bundle} if there exists a global homeomorphism \( E \cong B \times \mathbb{R}^n \) that respects the projection \( \pi \).

A \emph{bundle map} between two vector bundles \( \xi_1 = (E_1, B_1) \) and \( \xi_2 = (E_2,B_2) \) is a continuous map $G: E_1 \to E_2$ that maps each fiber \( \pi^{-1}(b) \) of \( \xi_1 \) to the corresponding fiber \( \pi^{-1}(b') \) of \( \xi_2 \) via a linear isomorphism. The induced map \( \hat{G}: B_1 \to B_2 \) describes how base points are mapped. Two vector bundles \( \xi_1 \) and \( \xi_2 \) are \emph{isomorphic} if there exists a homeomorphism \( G: E_1 \to E_2 \) that is a bundle map.

For two vector bundles $\xi_1 = (E_1,B)$ and $\xi_2 = (E_2,B)$ over a common base space $B$ with projections $\pi_1$ and $\pi_2$, their \emph{Whitney sum} $\xi_1 \oplus \xi_2$ is the vector bundle over $B$ whose fiber over each $b \in B$ is the direct sum $\pi_1^{-1}(b) \oplus \pi_2^{-1}(b)$ of the corresponding fibers.

The \emph{Möbius bundle} $\gamma_1^1 = (E,S^1)$ is defined as
$$E = \{(x,y) \in [0,1]\times\mathbb{R} \mid (0,y) \sim (1,-y)\},\quad S^1 = [0,1]/\{0 \sim 1\}, $$
where the projection map is defined as the projection onto the first coordinate. Up to isomorphism, this is the unique nontrivial rank-one vector bundle over the circle.

A key tool for studying vector bundles is the concept of \emph{characteristic classes}, which capture topological obstructions to a bundle being trivial (\citealt{steenrod1999topology}, Part 3). An important characteristic class is the Stiefel–Whitney class. For a vector bundle $\xi=(E,B)$ of finite rank, we define the \emph{Stiefel–Whitney classes} $\{w_k(\xi)\}_{k \in \mathbb{N} \cup \{0\}}$ of $\xi$ as the cohomology classes of $B$ with coefficients in $\mathbb{Z}/2$ that satisfy the following axioms:
\begin{itemize}
    \item The class $w_0(\xi) = 1 \in H^0(B;\mathbb{Z}/2)$ is the unit of the cohomology ring $H^*(B;\mathbb{Z}/2)$ for every vector bundle $\xi$, and $w_k(\xi) = 0$ whenever $\operatorname{rank}(\xi) < k$.
    \item A bundle map $f: E_1 \to E_2$ between vector bundles $\xi_1=(E_1,B_1)$ and $\xi_2=(E_2,B_2)$ satisfies $\hat{f}^*w_i(\xi_2)=w_i(\xi_1)$.
    \item Writing $w(\xi)=\sum_{i \geq 0} w_i(\xi) \in H^*(B;\mathbb{Z}/2)$ for the total class, the identity $w(\xi \oplus \eta)=w(\xi) \smile w(\eta)$ holds for any bundle $\eta$ over $B$.
    \item The Möbius bundle $\gamma_1^1$ over $S^1$ has $w_1(\gamma_1^1) \neq 0$.
\end{itemize}
Such classes exist and are uniquely determined by these axioms (\citealt{milnor1974characteristic}, Section 4). We call $w_k(\xi) \in H^k(B;\mathbb{Z}/2)$ the \emph{$k$-th Stiefel–Whitney class} of the vector bundle $\xi$. For a smooth manifold $M$, the \emph{$k$-th Stiefel–Whitney class of \( M \)} refers to the $k$-th Stiefel–Whitney class of its tangent bundle \( TM \), denoted $w_k(M)$.

\subsection{Wu formula}\label{wu}
In this subsection, we follow \citet{milnor1974characteristic}, Section 11, and study the Stiefel–Whitney classes of $M$. 

Let \( M \) be a closed connected \( n \)-dimensional smooth manifold. By Poincaré duality the cup-product pairing $H^k(M;\mathbb{Z}/2) \times H^{n-k}(M;\mathbb{Z}/2) \to H^n(M;\mathbb{Z}/2) \cong \mathbb{Z}/2$ is nondegenerate, so for each \( 0 \leq k \leq n \) there exists a unique cohomology class \( v_k \in H^k(M;\mathbb{Z}/2) \) satisfying
\[
v_k \smile x = \operatorname{Sq}^k(x)
\]
for every \( x \in H^{n-k}(M;\mathbb{Z}/2) \) (\citealt{milnor1974characteristic}, Section 11). This class \( v_k \) is called the \( k \)-th \emph{Wu class} of \( M \). It is the class dual to \( \operatorname{Sq}^k \) with respect to the cup-product pairing, and the defining condition is referred to as the \emph{Wu criterion}. Since the cup product and the Steenrod squares are linear over $\mathbb{Z}/2$, the criterion holds for every $x \in H^{n-k}(M;\mathbb{Z}/2)$ as soon as it holds on a basis, so it suffices to verify it on a basis of $H^{n-k}(M;\mathbb{Z}/2)$.

The \( k \)-th Stiefel–Whitney class of \( M \) is given by
\begin{equation}\label{wuformula}
    w_k = \sum_{i+j=k}\operatorname{Sq}^i(v_j).
\end{equation}
This equation, known as the \emph{Wu formula}, expresses the Stiefel–Whitney classes of \( M \) in terms of Wu classes and Steenrod squares. Since the cup product, the Steenrod squares, and Poincaré duality are homotopy invariants of the closed manifold \( M \), the Wu classes and the Stiefel–Whitney classes $w_k=\sum_{i+j=k}\operatorname{Sq}^i(v_j)$ they determine depend only on the homotopy type of \( M \). In particular, the Stiefel–Whitney classes of a closed manifold are topological invariants, independent of any smooth structure.

More generally, we call a space \( X \) a \emph{mod-2 Poincaré duality space of dimension \( n \)} if $H^n(X;\mathbb{Z}/2) \cong \mathbb{Z}/2$ and the cup-product pairing $H^k(X;\mathbb{Z}/2) \times H^{n-k}(X;\mathbb{Z}/2) \to H^n(X;\mathbb{Z}/2)$ is nondegenerate for every \( k \). Closed connected manifolds have this property by Poincaré duality (\citealt{hatcher2005algebraic}, Section 3.3), and so does every space homotopy equivalent to a closed manifold, since the pairing depends only on the cohomology ring.

\begin{lemma}\label{hptyeq}
    Assume that a map $f:X\to Y$ is a homotopy equivalence, and that \( X \) is a closed $n$-dimensional smooth manifold. Then, there exists a unique class \( v_k \in H^k(Y;\mathbb{Z}/2) \) satisfying
    \[
        v_k \smile y = \operatorname{Sq}^k(y)
    \]
    for every \( y \in H^{n-k}(Y;\mathbb{Z}/2) \). Moreover, \( f^* v_k \) is the \( k \)-th Wu class of \( X \), and
    \[
        f^* \left[\sum_{i+j=k} \operatorname{Sq}^i(v_j)\right]
    \]
    is the \( k \)-th Stiefel–Whitney class $w_k(X)$ of \( X \).

    \begin{proof}
        Since \( f \) is a homotopy equivalence, the induced homomorphism \( f^*:H^*(Y; \mathbb{Z}/2) \to H^*(X; \mathbb{Z}/2) \) is an isomorphism of graded-commutative algebras that preserves the cohomology ring structure. Therefore, there exists a unique class \( v_k \in H^k(Y; \mathbb{Z}/2) \) satisfying the Wu criterion for the given dimension $n$, and the isomorphism $f^*$ maps $v_k$ to the $k$-th Wu class of $X$.

        Denote \( w_k(X) \) as the \( k \)-th Stiefel–Whitney class of \( X \). Then by the Wu formula~\eqref{wuformula}, we have
        \[
            w_k(X) = \sum_{i+j=k} \operatorname{Sq}^i(f^* v_j) = \sum_{i+j=k} f^*(\operatorname{Sq}^i(v_j)) = f^* \left[\sum_{i+j=k} \operatorname{Sq}^i(v_j) \right].
        \]
    \end{proof}
\end{lemma}

\begin{remark}
    By Lemma~\ref{hptyeq}, when a space $Y$ is homotopy equivalent to a closed smooth manifold, its Wu and Stiefel–Whitney classes are determined by its cohomology ring over $\mathbb{Z}/2$ together with the Steenrod square operations. This is the basis of Algorithm~\ref{alg:sw_fixed}.
\end{remark}

\begin{definition}\label{wuswtype}
    For a given topological space \( X \) and integers \( 0 \leq k \leq n \), we define a \emph{$k$-th Wu class of dimension \( n \)} of \( X \) as a class \( v_{(k,n)} \) such that $v_{(k,n)} \smile x = \operatorname{Sq}^k(x)$ holds for every \( x \in H^{n-k}(X; \mathbb{Z}/2) \). For a general $X$ such a class need not exist or be unique; Proposition~\ref{prop:wusolve} characterizes its existence and uniqueness through the consistency of a linear system and the left-nondegeneracy of the cup-product pairing, respectively. When every $\{v_{(k,n)}\}_{0 \leq k \leq n}$ exists, we define a \emph{$k$-th Stiefel–Whitney class of dimension $n$} as $w_{(k,n)} = \sum_{i+j=k}\operatorname{Sq}^i(v_{(j,n)}).$

\end{definition}

\subsection{Persistence theory}\label{persistencetheory}

In this subsection, we follow \citet{contessoto2022cuplength}, with several modifications. We define a \emph{filtration of spaces} (or \emph{filtration}) to be a collection of topological spaces
\[
\mathbb{X} = \{X^r\}_{r \in \mathbb{R}}
\]
equipped with inclusion maps $\iota_r^{r'} : X^r \to X^{r'}$ for $r \leq r'$, satisfying $\iota_{r'}^{r''} \circ \iota_r^{r'} = \iota_r^{r''}$ for every $r \leq r' \leq r''$. 

Let \( (\mathbb{R},\leq) \) be a category having real numbers as objects, with a unique morphism \( r \to r' \) if and only if \( r \leq r' \), and let $\mathbf{Top}$ denote the category of compactly generated weak Hausdorff topological spaces. Then a filtration $\mathbb{X}$ can be viewed as a functor
\[
\mathbb{X} : (\mathbb{R},\leq) \to \mathbf{Top},\quad r \mapsto X^r,\quad (r,r')\mapsto (\iota_r^{r'} :X^r \to X^{r'})
\]
where each morphism is mapped to an inclusion map. For each $X^r$, we call the parameter $r$ a \emph{filtration scale}. If $\mathbb{X}$ is defined only for $r \in [s,t]$, we refer to $[s,t]$ as a \emph{filtration interval}. In this case, $\mathbb{X}$ is a functor defined on $([s,t],\leq)$, which is a full subcategory of \( (\mathbb{R},\leq) \) whose objects are the real numbers $r$ in $[s,t]$.

For a fixed field \( \mathbb{k} \), let $\mathbf{Vec}_\mathbb{k}$ denote the category whose objects are the $\mathbb{k}$-vector spaces and whose morphisms are \( \mathbb{k} \)-linear maps. The \emph{\( p \)-th persistent (singular) cohomology} \( H^p(\mathbb{X};\mathbb{k}) \) is defined as a functor
\[
H^p(\mathbb{X};\mathbb{k}) : ([s,t],\leq)^{\operatorname{op}} \to \mathbf{Vec}_\mathbb{k},
\]
where
\[
r \mapsto H^p(X^r;\mathbb{k}), \quad (r \leq r') \mapsto (\iota_r^{r'*} : H^p(X^{r'};\mathbb{k}) \to H^p(X^r;\mathbb{k})).
\]
Equivalently, it can be viewed as the composition of functors:
\[
H^p(\mathbb{X};\mathbb{k}) : ([s,t],\leq)^{\operatorname{op}} \xrightarrow{\mathbb{X}^{\operatorname{op}}} \mathbf{Top}^{\operatorname{op}} \xrightarrow{H^p(\cdot;\mathbb{k})} \mathbf{Vec}_\mathbb{k}.
\]
If each $H^p(X^r;\mathbb{Z}/2)$ is finite dimensional, the persistent cohomology decomposes as a direct sum of interval modules (\citealt{crawley2015decomposition}). The collection of intervals that represent the birth and death of each linearly independent cohomology class is called a \emph{persistent barcode}, and an individual interval is a \emph{bar}.

We define a \emph{\( p \)-th persistent cohomology class} of \( \mathbb{X} \) over any interval \( [s,t] \) in its index set to be a family \( \mathbb{x}=\{x^r\}_{s\leq r\leq t} \) with \( x^r \in H^p(X^r;\mathbb{k}) \) such that $\iota_{r}^{r'*} (x^{r'})= x^{r}$ for every \( s \leq r \leq r' \leq t \).
By the functoriality of the cup product, the cohomology ring functor induces
\[
H^*(\cdot; \mathbb{k}) : \mathbf{Top}^{\operatorname{op}} \to \mathbf{GrAlg}_\mathbb{k},
\]
where $\mathbf{GrAlg}_\mathbb{k}$ denotes the category of graded-commutative algebras over \( \mathbb{k} \). Consequently, we obtain the composition
\[
H^*(\mathbb{X};\mathbb{k}) : ([s,t],\leq)^{\operatorname{op}} \xrightarrow{{\mathbb{X}}^{\operatorname{op}}} \mathbf{Top}^{\operatorname{op}} \xrightarrow{H^*(\cdot;\mathbb{k})} \mathbf{GrAlg}_\mathbb{k}
\]
which is called a \emph{persistent graded-commutative algebra over $\mathbb{k}$}.

\section{Persistent Stiefel–Whitney class}\label{sec:perssw}

Throughout this section (co)homology is taken with $\mathbb{Z}/2$ coefficients, and
$\mathbb{X}=\{X^r\}_{r\in[s,t]}$ denotes a filtration of spaces in the sense of
Section~\ref{persistencetheory}. We introduce persistent analogues of the Wu and
Stiefel–Whitney classes, defined through the cup product and the Steenrod squares on the
persistent cohomology of $\mathbb{X}$. Our main results show that if the space $X^t$ at the top of
the filtration is homotopy equivalent to a smooth closed manifold $M$, then, under a compatibility
hypothesis on the filtration made precise below, these classes exist and represent the
Stiefel–Whitney classes of $M$.

\subsection{Persistent Wu formula}
\begin{definition}\label{perswu}
    Let $\mathbb{X} = \{X^r\}_{r \in [s,t]}$ be a filtration of spaces defined on $[s,t]$, and let $0 \leq k \leq n$. A persistent cohomology class
    \[
    \mathbb{v}_{(k,n)} = \{v_{(k,n)}^r\}_{r \in [s,t]} \in H^k(\mathbb{X};\mathbb{Z}/2)
    \]
    is called the \emph{$k$-th persistent Wu class of dimension $n$} if it satisfies the following two conditions.
    \begin{enumerate}
        \item[(i)] \emph{(Persistent Wu criterion)} For every $r \in [s,t]$ and every $x \in H^{n-k}(X^r;\mathbb{Z}/2)$,
        \[
        v_{(k,n)}^r \smile x = \operatorname{Sq}^k(x).
        \]
        \item[(ii)] $\mathbb{v}_{(k,n)}$ is the unique persistent cohomology class satisfying~(i).
    \end{enumerate}
    At an individual scale $r$, the equation in~(i) may admit several solutions $v^r \in H^k(X^r;\mathbb{Z}/2)$; condition~(ii) requires uniqueness only among the functorially compatible families $\{v^r\}_{r \in [s,t]}$ (those with $\iota_r^{r'*}v^{r'} = v^r$), not at each scale separately.
\end{definition}

\begin{remark}\label{endpoint}
    In general, the criterion in Definition~\ref{perswu}(i) can be verified separately on the classes
    representing the persistent cohomology bars. Choose classes $x_i^r\in H^{n-k}(X^r)$
    representing the bars, compatible under pullback within each bar, so that the values
    $x_i^r$ of the bars alive at $r$ form a basis of $H^{n-k}(X^r)$. Since the cup product and the
    Steenrod squares commute with the induced maps $\iota_r^{r'*}$, the equality
    $v_{(k,n)}^r\smile x_i^r=\operatorname{Sq}^k(x_i^r)$ at one scale of a bar propagates
    by pullback to every lower scale in that bar.

    In the finite simplicial filtrations used in Algorithm~\ref{alg:pers_sw}, each simplex is
    included at its assigned filtration value. The bars restricted to $[s,t]$ therefore have
    the form $[s_i,t_i)$ with $s\leq s_i<t_i\leq t$, or $[s_i,t]$ with $s\leq s_i\leq t$
    if they contain $t$. For a bar $[s_i,t_i)$, it suffices to check the equality at the last
    stage of the filtration before $t_i$, when the bar is still nonzero. For a bar containing
    $t$, it suffices to check it at $t$. By linearity, these finitely many checks establish
    the criterion throughout $[s,t]$.
\end{remark}

\begin{definition}\label{persswtype}
    Suppose the persistent Wu classes $\mathbb{v}_{(0,n)},\ldots,\mathbb{v}_{(k,n)}$ of dimension $n$ all exist. The \emph{$k$-th persistent Stiefel–Whitney class of dimension $n$} is defined by the \emph{persistent Wu formula}
    \[
    \mathbb{w}_{(k,n)}= \{w_{(k,n)}^r\}_{r \in [s,t]}, \qquad w_{(k,n)}^r = \sum_{i+j=k}\operatorname{Sq}^{i}\big(v_{(j,n)}^r\big).
    \]
    Because each $\mathbb{v}_{(j,n)}$ is a persistent cohomology class and the Steenrod squares commute with the induced maps $\iota_r^{r'*}$, the family $\mathbb{w}_{(k,n)}$ is again a persistent cohomology class over $[s,t]$.
\end{definition}

\begin{proposition}\label{swequal}
Let $\mathbb{X}=\{X^r\}_{r\in[s,t]}$ be a filtration with a homotopy equivalence $f\colon X^t\to M$
onto a closed smooth $n$-dimensional manifold $M$, and suppose $\iota_r^{t*}\colon H^*(X^t)\to H^*(X^r)$
is an isomorphism for every $r\in[s,t]$. Then for each $1\le k\le n$, the persistent Wu and
Stiefel–Whitney classes of dimension $n$ exist on $[s,t]$, with $w_{(k,n)}^r=\iota_r^{t*}f^*w_k(M)$.
In particular, at every scale $r\in[s,t]$, the isomorphism $\iota_r^{t*}f^*:H^k(M)\to H^k(X^r)$ identifies $w_k(M)$ with $w_{(k,n)}^r$.
\begin{proof}
Let $g\colon M\to X^t$ be a homotopy inverse of $f$. By Lemma~\ref{hptyeq} applied to $g$, for each
$k$ the space $X^t$ carries a unique Wu class $v_{(k,n)}^t\in H^k(X^t)$, and
$w_{(k,n)}^t:=\sum_{i+j=k}\operatorname{Sq}^i(v_{(j,n)}^t)$ satisfies $g^*w_{(k,n)}^t=w_k(M)$, hence
$w_{(k,n)}^t=f^*w_k(M)$. Set $v_{(k,n)}^r:=\iota_r^{t*}v_{(k,n)}^t$. Since $\iota_r^{t*}$ is a ring
isomorphism commuting with the Steenrod squares, every $x\in H^{n-k}(X^r)$ equals $\iota_r^{t*}\tilde x$
for a unique $\tilde x\in H^{n-k}(X^t)$, so
\[
v_{(k,n)}^r\smile x=\iota_r^{t*}\big(v_{(k,n)}^t\smile\tilde x\big)=\iota_r^{t*}\operatorname{Sq}^k(\tilde x)=\operatorname{Sq}^k(x).
\]
Hence $\{v_{(k,n)}^r\}_{r\in[s,t]}$ is the persistent Wu class of dimension $n$, and the persistent Wu
formula gives $w_{(k,n)}^r=\iota_r^{t*}w_{(k,n)}^t=\iota_r^{t*}f^*w_k(M)$.
\end{proof}
\end{proposition}

Proposition~\ref{swequal} requires the induced maps $\iota_r^{t*}$ to be isomorphisms. This condition holds, in particular, on an interval over which the filtration inclusions are homotopy equivalences. We next consider filtrations with additional wedge summands and assume that all positive Steenrod squares vanish on their reduced cohomology. Under the hypotheses of the following theorem, the persistent Stiefel–Whitney classes recover those of $TM$.

\begin{theorem}\label{thm:wedge}
Let $\mathbb{X}=\{X^r\}_{r\in[s,t]}$ be a filtration and $M$ a closed smooth $n$-dimensional
manifold. Suppose that for each $r$, there exists a homotopy equivalence $X^r\to M\vee Y^r$, and let
$\rho_r\colon X^r\to M$ be its composite with the canonical retraction $M\vee Y^r\to M$ collapsing $Y^r$ to the wedge point.
Assume that $\rho_{r'}\circ\iota_r^{r'}\simeq\rho_r$
for $r\le r'$. Suppose further that $\operatorname{Sq}^i=0$ on $\widetilde H^*(Y^r;\mathbb{Z}/2)$ for
every $i>0$ and every $r$. Put $v_k^r:=\rho_r^*v_k(M)$, where $v_k(M)\in H^k(M;\mathbb{Z}/2)$ is the
$k$-th Wu class of $M$.
\begin{enumerate}
\item[(1)] For each $0\le k\le n$, the persistent cohomology class $\mathbb{v}_k=\{v_k^r\}_{r\in[s,t]}$
satisfies the persistent Wu criterion of Definition~\ref{perswu}. In general it need not be the
only persistent cohomology class that does so.
\item[(2)] If in addition $X^t\simeq M$, then $\mathbb{v}_k$ is the unique such class, so the $k$-th
persistent Wu class $\mathbb{v}_{(k,n)}$ of dimension $n$ exists and equals $\mathbb{v}_k$.
Consequently the persistent Wu formula gives $\mathbb{w}_{(k,n)}=\{\rho_r^*w_k(M)\}_{r\in[s,t]}$,
which represents the $k$-th Stiefel–Whitney class $w_k(M)$.
\end{enumerate}
\end{theorem}

\begin{proof}
Since $\rho_{r'}\circ\iota_r^{r'}\simeq\rho_r$, these maps induce the same homomorphism on cohomology,
so $\iota_r^{r'*}v_k^{r'}=\rho_r^*v_k(M)=v_k^r$ and $\mathbb{v}_k$ is a persistent cohomology class on $[s,t]$. The
equivalence $X^r\simeq M\vee Y^r$ gives $H^*(X^r)=\rho_r^*H^*(M)\oplus\widetilde H^*(Y^r)$,
where $\widetilde H^*(Y^r)$ is identified with the corresponding summand in $H^*(X^r)$.
Cup products between $\rho_r^*\widetilde H^*(M)$ and $\widetilde H^*(Y^r)$ vanish and, by naturality,
each $\operatorname{Sq}^i$ preserves the two summands.

(1) For $k=0$, the criterion holds because $v_0^r=\rho_r^*1=1$ and $\operatorname{Sq}^0$ is the identity.
For $k>0$, write $x=\rho_r^*x_M+x_Y\in H^{n-k}(X^r)$ with $x_M\in H^{n-k}(M)$ and $x_Y\in\widetilde H^{n-k}(Y^r)$.
The cup product of $v_k^r=\rho_r^*v_k(M)$ with $x_Y$ vanishes, so by the Wu criterion on $M$,
\[
v_k^r\smile x=\rho_r^*v_k(M)\smile\rho_r^*x_M=\rho_r^*\big(v_k(M)\smile x_M\big)=\rho_r^*\operatorname{Sq}^k(x_M),
\]
while $\operatorname{Sq}^k(x)=\rho_r^*\operatorname{Sq}^k(x_M)+\operatorname{Sq}^k(x_Y)=\rho_r^*\operatorname{Sq}^k(x_M)$
by naturality and the hypothesis on $Y^r$. Hence $\mathbb{v}_k$ satisfies the persistent Wu criterion.

(2) Suppose in addition $X^t\simeq M$. Let $\{v^r\}_{r\in[s,t]}$ be any persistent cohomology class satisfying the Wu criterion.
At scale $t$, the class $v^t$ satisfies $v^t\smile x=\operatorname{Sq}^k(x)$ for every $x\in H^{n-k}(X^t)$.
Since $X^t\simeq M$, Lemma~\ref{hptyeq} implies $v^t=v_k^t$.
For every $r\in[s,t]$, compatibility with the filtration gives
\[
v^r=\iota_r^{t*}v^t=\iota_r^{t*}v_k^t=v_k^r.
\]
Thus $\mathbb{v}_k$ is the unique such class, and $\mathbb{v}_{(k,n)}=\mathbb{v}_k$ exists. The
same holds in each degree $j\le k$, giving $v_{(j,n)}^r=\rho_r^*v_j(M)$, so Definition~\ref{persswtype}
yields
\[
w_{(k,n)}^r=\sum_{i+j=k}\operatorname{Sq}^i(v_{(j,n)}^r)
=\rho_r^*\sum_{i+j=k}\operatorname{Sq}^i(v_j(M))=\rho_r^*w_k(M),
\]
which represents $w_k(M)$ by the Wu formula on $M$.
\end{proof}

Theorem~\ref{thm:wedge} allows additional cohomology classes from the summands $Y^r$ at scales $r<t$, whereas Proposition~\ref{swequal} requires the induced maps $\iota_r^{t*}$ to be isomorphisms. Under the hypotheses of the theorem, the persistent Stiefel–Whitney classes are still given by $\mathbb{w}_{(k,n)}=\{\rho_r^*w_k(M)\}_{r\in[s,t]}$.

\begin{remark}\label{rem:noise}
The vanishing condition on the Steenrod squares is motivated by results in random topology.
In the sparse regime of random Čech and Vietoris–Rips complexes, components homotopy equivalent to spheres give the leading contribution to the expected Betti numbers in positive degrees \citep{kahle2011random,bobrowski2018topology}.
For samples affected by weak noise near $M$, these results suggest modelling the additional topology by summands $Y^r$ that are homotopy equivalent to wedges of spheres.
All positive Steenrod squares vanish on the reduced cohomology over $\mathbb{Z}/2$ of a sphere because they raise degree and this cohomology is nonzero in only one degree. Naturality gives the same conclusion for a wedge of spheres.
Consequently, under the remaining hypotheses of Theorem~\ref{thm:wedge}, these summands can contribute additional cohomology classes while the persistent Stiefel–Whitney classes still recover those of $TM$.
\end{remark}

\section{Čech estimation}\label{sec:cech}

Under its hypotheses, Theorem~\ref{thm:wedge} guarantees the existence and uniqueness of persistent Stiefel--Whitney classes of dimension $n$ even when additional wedge summands modelling noise are present at scales $r<t$. In this section, we consider samples drawn from $M$ and show that, with high probability, the Čech and alpha complexes are homotopy equivalent to $M$ throughout an interval of scales. Using the functorial Nerve Lemma, we also show that the filtration inclusions are homotopy equivalences and apply Proposition~\ref{swequal} to establish the existence and uniqueness of these classes.

\subsection{Filtrations from point clouds}
A detailed explanation and visualization can be found in \citet{edelsbrunner2010computational}. Given a point cloud \( P = \{p_i\}_{i=1}^N \subset \mathbb{R}^D \) in general position, there exist several methods to construct filtrations that capture its topological structure at different scales.

We use closed balls in the definitions of the Čech and alpha complexes. For $r\geq 0$, we define a \emph{Čech complex} $\check{C}(P,r)$ as the nerve of the collection of closed balls $\{B_{\mathbb{R}^D}(p_i,r)\}_{p_i \in P}$, where $B_{\mathbb{R}^D}(p_i,r)$ denotes the closed ball of radius $r$ centered at $p_i$. For \( 0 \leq r \leq r' \), there exists a natural simplicial inclusion $\check{C}(P, r) \subset \check{C}(P, r'),$
which ensures that the family of simplicial complexes \( \{\check{C}(P, r)\}_{r \geq 0} \) forms a filtration, known as the \emph{Čech filtration}.

A \emph{Delaunay simplex} of \( P \subset \mathbb{R}^D \) is the convex hull of a subset \( \sigma \subseteq P \) for which there is a closed ball whose bounding sphere passes through the points of \( \sigma \) and whose interior contains no point of \( P \). The collection of all Delaunay simplices is the \emph{Delaunay complex} of \( P \). For \( r \geq 0 \), the \emph{alpha complex} \( \alpha(P, r) \) is the subcomplex consisting of those Delaunay simplices that admit such an empty circumscribing ball of radius at most \( r \), and \( \{\alpha(P, r)\}_{r \geq 0} \) is a filtration, the \emph{alpha filtration}. For each \( r \geq 0 \), both \( \check{C}(P,r) \) and \( \alpha(P, r) \) are homotopy equivalent to the union of balls \( \bigcup_{p_i \in P} B_{\mathbb{R}^D}(p_i,r) \), and therefore to one another \citep{edelsbrunner2010computational}.

\subsection{Probabilistic argument}
To ensure the existence and effective computation of persistent Stiefel–Whitney classes with high probability, we rely on the Nerve Lemma and the main theorem from \citet{niyogi2008finding}.

\begin{lemma}[\citealt{edelsbrunner2010computational}, Nerve Lemma]\label{nerve}
Let \(\mathcal{U} = \{U_i\}_{i \in I}\) be a finite cover of \(X\subseteq\mathbb{R}^D\) by closed convex sets. Then the nerve of \(\mathcal{U}\) is homotopy equivalent to \(X\).
\end{lemma}
\begin{remark}
The balls and their intersections with Voronoi cells are closed and convex. Thus, by applying Lemma~\ref{nerve}, the Čech complex $\check{C}(P,r)$ and the alpha complex $\alpha(P,r)$ are homotopy equivalent to the union of balls $\bigcup_{p_i \in P} B_{\mathbb{R}^D}(p_i,r).$ 
\end{remark}

The ordinary Nerve Lemma does not address compatibility with the inclusion
maps of a filtration. We therefore use the functorial version proved by
\citet{bauer2023unified}.

\begin{lemma}[Functorial Nerve Lemma]\label{functorialnerve}
Let $P\subset\mathbb{R}^D$ be finite and in general position, let
$U^r=\bigcup_{p\in P}B_{\mathbb{R}^D}(p,r)$, and let $X^r$ denote either
$\check C(P,r)$ or $\alpha(P,r)$. For $r\leq r'$, there are homotopy
equivalences $\eta_r:|X^r|\to U^r$ and $\eta_{r'}:|X^{r'}|\to U^{r'}$ for
which
\[
\eta_{r'}\circ\iota_r^{r'}\simeq j_r^{r'}\circ\eta_r,
\]
where $\iota_r^{r'}:|X^r|\hookrightarrow|X^{r'}|$ and
$j_r^{r'}:U^r\hookrightarrow U^{r'}$ are the filtration inclusions.
Consequently, if $j_r^{r'}$ is a homotopy equivalence, then so is
$\iota_r^{r'}$.
\end{lemma}

\begin{proof}
For the Čech filtration, use the nested covers
$\{B_{\mathbb{R}^D}(p,r)\}_{p\in P}$ of $U^r$. For the alpha filtration,
use the nested restricted Voronoi covers
$\{B_{\mathbb{R}^D}(p,r)\cap V_p\}_{p\in P}$, whose nerves are the alpha
complexes. Here $V_p$ is the Voronoi cell of $p$. These sets cover $U^r$
because every point of $U^r$ has a nearest point $p\in P$ and belongs to
$B_{\mathbb{R}^D}(p,r)\cap V_p$.
All nonempty finite intersections in either cover are closed and convex, and
each cover element at scale $r$ is contained in the corresponding one at
scale $r'$. Theorem~3.9 of \citet{bauer2023unified} gives natural homotopy
equivalences from the blowup complex of each cover to $U^r$ and to $|X^r|$.
Composing the former with a homotopy inverse of the latter defines
$\eta_r$. Naturality implies that the displayed diagram commutes up to homotopy. The
final assertion follows from the two-out-of-three property for homotopy
equivalences.
\end{proof}

Let \( M \) be a compact smooth submanifold embedded in \( \mathbb{R}^D \). The \emph{reach} of \( M \), denoted by \( \tau(M) \), is defined as  
\begin{align*}
    \tau (M) = \sup \big\{ r \geq 0 \,\big|\, &\text{every point within distance } r \text{ of } M \\
    &\text{has a unique nearest point on } M \big\}.
    \end{align*}
Intuitively, the reach reflects both the curvature of $M$ and the width of its narrowest bottleneck. A larger reach indicates a smoothly curved manifold, whereas a small reach suggests sharp bends or potential self-intersections (see \citealt{federer1959curvature, aamari2019estimating}).

\begin{theorem}[\citealt{niyogi2008finding}]\label{nsw} Let $M$ be a compact $n$-dimensional submanifold of $\mathbb{R}^D$ with reach at least \( \tau \). Let \( P = \{p_1,\ldots,p_N\} \) be a set of \( N \) points drawn i.i.d. according to the uniform probability measure on \( M \). Let
\[
U = \bigcup_{p_i \in P} B_{\mathbb{R}^D}(p_i,\varepsilon)
\]
where $B_{\mathbb{R}^D}(p_i,\varepsilon)$ is a closed ball of radius $\varepsilon$ centered at $p_i$. Denote by \( B_{\mathbb{R}^n}(0,\varepsilon) \) the closed ball in \( \mathbb{R}^n \) of radius \( \varepsilon \) centered at the origin. Then, for any \( \delta > 0 \), with probability at least \( 1 - \delta \), the set $U$ deformation retracts to $M$ provided
\[
N \geq \beta_1 \left(\log(\beta_2)+ \log\left(\frac{1}{\delta}\right)\right)
\]
and $\varepsilon < \frac{\tau}{2}$, where
\[
\begin{aligned}
\beta_1 &=\frac{\operatorname{vol}(M)}{
  {\cos^n(\theta_1)\operatorname{vol}(B_{\mathbb{R}^n}(0,\varepsilon/4))}},
&\theta_1&=\arcsin{\frac{\varepsilon}{8\tau}},\\
\beta_2 &=\frac{\operatorname{vol}(M)}{
  {\cos^n(\theta_2)\operatorname{vol}(B_{\mathbb{R}^n}(0,\varepsilon/8))}},
&\theta_2&=\arcsin{\frac{\varepsilon}{16\tau}}.
\end{aligned}
\]
\end{theorem}

The original result uses open balls. On the event that $P$ is
$\varepsilon/2$-dense, the homotopy along line segments to the nearest point
on $M$ preserves the union of open balls of every radius slightly larger
than $\varepsilon$. Since $P$ is finite, the intersection of these unions is
$U$, so the same homotopy preserves $U$.

\begin{lemma}[Uniform reconstruction over a scale interval]\label{uniformnsw}
Under the hypotheses of Theorem~\ref{nsw}, suppose that
$0<\varepsilon<\tau/2$ and that the stated sample size bound holds. Then with
probability at least $1-\delta$, both of the following statements hold
simultaneously whenever $\varepsilon\leq r\leq r'\leq\tau/2$.
\begin{enumerate}
    \item $M$ is a deformation retract of $U^r$.
    \item The inclusion $j_r^{r'}:U^r\hookrightarrow U^{r'}$ is a homotopy
    equivalence.
\end{enumerate}
\end{lemma}

\begin{proof}
By Proposition~3.2 of \citet{niyogi2008finding}, the probability that $P$ is
$\varepsilon/2$-dense in $M$ is at least $1-\delta$. If $P$ is
$\varepsilon/2$-dense, then it is $r/2$-dense for every
$r\geq\varepsilon$.
Moreover, $r\leq\tau/2<\sqrt{3/5}\,\tau$, so Proposition~3.1 of the same
paper, with the preceding argument for closed balls, implies simultaneously for every $r\in[\varepsilon,\tau/2]$ that the
inclusion $a_r:M\hookrightarrow U^r$ is a homotopy equivalence, indeed a
deformation retract. For $r\leq r'$ we have
$a_{r'}=j_r^{r'}\circ a_r$. Since $a_r$ and $a_{r'}$ are homotopy
equivalences, the two-out-of-three property implies that $j_r^{r'}$ is a
homotopy equivalence.
\end{proof}

Let  \( \mathcal{M} \) denote the family of compact smooth \( n \)-dimensional manifolds \( M \) generically embedded in \( \mathbb{R}^D \), satisfying \( \operatorname{reach}(M) \geq \tau \) and \( \operatorname{vol}_n(M) \leq V_n \). Denote by \( B_{\mathbb{R}^n}(0,\varepsilon) \) the closed ball in \( \mathbb{R}^n \) of radius \( \varepsilon \) centered at the origin.

\begin{corollary}\label{main}
    Let $M\in\mathcal{M}$, let $0<\varepsilon<\tau/2$, and let
    $P=\{p_i\}_{i=1}^N\subset M\subset\mathbb{R}^D$ be sampled independently
    from the uniform probability measure on $M$. Set
    $U^r=\bigcup_{1\leq i\leq N}B_{\mathbb{R}^D}(p_i,r)$. Suppose that
    \[
    N\geq\beta_1\left(\log(\beta_2)+
    \log\left(\frac{1}{\delta}\right)\right),
    \]
    where
    \[
    \begin{aligned}
    \beta_1 &=\frac{V_n}{
      {\cos^n(\theta_1)\operatorname{vol}(B_{\mathbb{R}^n}(0,\varepsilon/4))}},
    &\theta_1&=\arcsin{\frac{\varepsilon}{8\tau}},\\
    \beta_2 &=\frac{V_n}{
      {\cos^n(\theta_2)\operatorname{vol}(B_{\mathbb{R}^n}(0,\varepsilon/8))}},
    &\theta_2&=\arcsin{\frac{\varepsilon}{16\tau}}.
    \end{aligned}
    \]
    Then, with probability at least $1-\delta$, the Čech and alpha
    filtrations both admit persistent Stiefel--Whitney classes
    $\{\mathbb{w}_k\}_{k=1}^n$ of dimension $n$ on $[\varepsilon,\tau/2]$. If $X^r$ denotes
    either of these complexes, the homotopy equivalences $f^r:X^r\to M$ may
    be chosen so that $f^{r'}\circ\iota_r^{r'}\simeq f^r$ for $r\leq r'$ and
    \[
    f^{r*}w_k(M)=w_k^r,
    \qquad
    \mathbb{w}_k=\{w_k^r\}_{r\in[\varepsilon,\tau/2]}.
    \]
    Here $w_k(M)$ is the $k$-th Stiefel–Whitney class of $M$. In other words, for each $k$, there exists a (not necessarily nontrivial) persistent cohomology class with a lifespan of at least \( \frac{\tau}{2} - \varepsilon \) that represents \( w_k(M) \).
\end{corollary}

\begin{proof}
    The genericity assumption ensures that $P$ is in general position almost surely.
    Since $\operatorname{vol}_n(M)\leq V_n$, the sample size assumption
    implies the bound in Lemma~\ref{uniformnsw}. It follows that, with
    probability at least $1-\delta$, the inclusion
    $U^r\hookrightarrow U^{r'}$ is a homotopy equivalence whenever
    $\varepsilon\leq r\leq r'\leq\tau/2$. Hence
    Lemma~\ref{functorialnerve} implies that the filtration inclusion
    $\iota_r^{r'}:X^r\hookrightarrow X^{r'}$ is also a homotopy equivalence,
    for either the Čech or the alpha filtration.

    Put $T=\tau/2$ and choose a homotopy equivalence $f^T:X^T\to M$. For
    $r\in[\varepsilon,T]$, define $f^r=f^T\circ\iota_r^T$. Each $f^r$ is a
    homotopy equivalence and $f^{r'}\circ\iota_r^{r'}=f^r$. In particular,
    each induced map $\iota_r^{T*}:H^*(X^T;\mathbb{Z}/2)\to
    H^*(X^r;\mathbb{Z}/2)$ is an isomorphism. Proposition~\ref{swequal}
    therefore gives
    \[
    w_k^r=\iota_r^{T*}f^{T*}w_k(M)=f^{r*}w_k(M),
    \]
    and the classes $\{w_k^r\}_{r\in[\varepsilon,T]}$ form the asserted
    persistent Stiefel--Whitney class.
\end{proof}

\section{Algorithm}\label{sec:algorithm}

\subsection{Algorithm for computing persistent Stiefel–Whitney classes}

We present two specific computational algorithms: one for computing Stiefel–Whitney classes of dimension $n$ at a fixed filtration scale (Algorithm~\ref{alg:sw_fixed}) and another for tracking their persistence across a filtration (Algorithm~\ref{alg:pers_sw}).

The only step of Algorithm~\ref{alg:sw_fixed} that is not a standard (co)homology computation is the determination of the Wu classes, which reduces to solving a linear system over $\mathbb{Z}/2$ by Gaussian elimination.

\begin{proposition}\label{prop:wusolve}
Fix bases $\{\alpha_i\}_{i=1}^{a}$, $\{x_j\}_{j=1}^{b}$, and $\{\gamma_l\}_{l=1}^{c}$ of $H^k(X^t;\mathbb{Z}/2)$, $H^{n-k}(X^t;\mathbb{Z}/2)$, and $H^n(X^t;\mathbb{Z}/2)$ respectively, read off from the representative cocycles of the persistent cohomology, and let $L\in(\mathbb{Z}/2)^{bc\times a}$ and $q\in(\mathbb{Z}/2)^{bc}$ have entries
\[
L_{(jl),i}=\langle\alpha_i\smile x_j,\gamma_l\rangle,\qquad q_{jl}=\langle\operatorname{Sq}^k(x_j),\gamma_l\rangle,
\]
where $\langle\,\cdot\,,\gamma_l\rangle$ denotes the $\gamma_l$-coordinate. Then $v=\sum_i\xi_i\alpha_i$ is a $k$-th Wu class of dimension $n$ of $X^t$ if and only if $L\xi=q$. Consequently the system is consistent if and only if the $k$-th Wu class of dimension $n$ exists in $H^k(X^t;\mathbb{Z}/2)$, and its solution is unique if and only if the cup-product pairing $H^k(X^t)\times H^{n-k}(X^t)\to H^n(X^t)$ is left-nondegenerate. In particular, if $X^t$ is a mod-2 Poincaré duality space of dimension $n$, then the Wu class exists and is unique.
\end{proposition}

\begin{proof}
Expanding $v=\sum_i\xi_i\alpha_i$ and reading off the $\gamma_l$-coordinate gives $\langle v\smile x_j,\gamma_l\rangle=\sum_i\langle\alpha_i\smile x_j,\gamma_l\rangle\,\xi_i$, so the equalities $v\smile x_j=\operatorname{Sq}^k(x_j)$ for all $j,l$ are precisely the system $L\xi=q$. Since $\{x_j\}$ is a basis, these are equivalent to $v\smile x=\operatorname{Sq}^k(x)$ for every $x\in H^{n-k}(X^t)$, so $v$ is a $k$-th Wu class of dimension $n$ if and only if $L\xi=q$. The system is therefore consistent exactly when such a $v$ exists, and $L\xi=0$ forces $\xi=0$ exactly when no nonzero $v$ has $v\smile x=0$ for all $x$, which is the left-nondegeneracy of the cup-product pairing. If $X^t$ is a mod-2 Poincaré duality space of dimension $n$, this pairing is perfect, so $v\mapsto(x\mapsto v\smile x)$ is a bijection onto the functionals on $H^{n-k}(X^t)$, giving both existence and uniqueness.
\end{proof}

\begin{algorithm}[H]
\caption{Computation of Stiefel–Whitney Classes of dimension $n$ at filtration Scale $t$}
\label{alg:sw_fixed}
\begin{algorithmic}[1]
\Require Finite sampled points $P=\{p_i\}_{i=1}^N \subset \mathbb{R}^D$ from an underlying manifold $M$.
\Require Filtration scale $t$.
\Require Estimated intrinsic dimension $n$ of the manifold $M$.
\Ensure Stiefel–Whitney classes $\{w_{(k,n)}\}_{k=1}^n$ of dimension $n$ at the filtration scale $t$ in Definition~\ref{wuswtype}.

\State Construct a Čech or alpha filtration $\mathbb{X} = \{X^r\}_{r \geq 0}$ from $P$.
\State Compute the cohomology $H^*(X^t; \mathbb{Z}/2)$ and extract representative cocycles.
\For{each $0 \leq k \leq n$}
    \State Assemble and solve the linear system $L\xi = q$ of Proposition~\ref{prop:wusolve} over $\mathbb{Z}/2$, and set $v_{(k,n)} = \sum_i \xi_i \alpha_i$. If the system is inconsistent, report that the Wu class does not exist.
\EndFor
\For{each $1 \leq k \leq n$}
    \State Compute the $k$-th Stiefel–Whitney class of dimension $n$ using the Wu formula~\ref{wuformula}:
    \[
        w_{(k,n)} = \sum_{i+j=k} \operatorname{Sq}^i(v_{(j,n)}).
    \]
\EndFor
\State \Return $\{w_{(k,n)}\}_{k=1}^n$.
\end{algorithmic}
\end{algorithm}

\begin{algorithm}[H]
\caption{Computation of Persistent Stiefel–Whitney Classes of dimension $n$}
\label{alg:pers_sw}
\begin{algorithmic}[1]
\Require Finite sampled points $\{p_i\}_{i=1}^N \subset \mathbb{R}^D$ from an underlying manifold $M$.
\Require Estimated intrinsic dimension $n$ of underlying manifold $M$.
\Require Filtration interval $[s, t]$.
\Require The Wu class of dimension $n$ at scale $t$ is unique in every degree.
\Ensure Persistent Stiefel–Whitney classes $\{\mathbb{w}_{(k,n)}\}_{k=1}^n$ of dimension $n$ over $[s, t]$ if they exist.

\State Use Algorithm~\ref{alg:sw_fixed} to compute the filtration $\mathbb{X}$ and the Wu classes $\{v_{(i,n)}^t\}$ of dimension $n$ at filtration scale $t$.
\State Restrict the persistent cohomology to $[s,t]$. For each bar $I_l$ of degree at most $n$ with $t\notin I_l$, let $r_l$ be the last filtration value at which the bar is nonzero, and extract a representative cocycle $x_l$ at $r_l$.
\For{each $l$}
    \State For $k = n-\deg x_l$, verify the Wu criterion at $r_l$ as in Remark~\ref{endpoint}.
    \[
        v_{(k,n)}^{r_l} = \iota_{r_l}^{t*} v_{(k,n)}^t,
        \qquad
        v_{(k,n)}^{r_l}\smile x_l = \operatorname{Sq}^k(x_l).
    \]
\EndFor
\If{valid for all $l$}
    \State Compute persistent Stiefel–Whitney classes of dimension $n$ using the persistent Wu formula~\ref{persswtype}:
    \[
        \mathbb{w}_{(k,n)} = \left\{w_{(k,n)}^r\right\}_{r \in [s, t]}, \quad w_{(k,n)}^r = \iota_r^{t*}\!\left[\sum_{i+j=k} \operatorname{Sq}^i(v_{(j,n)}^t)\right].
    \]
\EndIf
\State \Return $\{\mathbb{w}_{(k,n)}\}_{k=1}^n$.
\end{algorithmic}
\end{algorithm}

\subsection{Algorithms for cohomology operations}
Given a simplicial complex \( X = \{\sigma_i\} \), any cocycle \( c \in Z^p(X; \mathbb{Z}/2) \) can be expressed as  
\[
c = \sum_{j \in J} \sigma_j^*,
\]
where each \( \sigma_j^* \) is the dual cochain corresponding to the \( p \)-dimensional simplex \( \sigma_j \). For convenience, we write \(c = \sum_{j\in J} \sigma_j \) or $\{\sigma_j\}_{j \in J}$ to represent its underlying support.

For the computation of the cup products of two cochains (Algorithm~\ref{alg:cup_product}), we refer to Algorithm~1 in \citet{contessoto2021persistent}. To compute Steenrod squares, we employ Theorem \ref{newsq}, which utilizes the cup-\( i \) coproduct (see Algorithm~\ref{alg:cupico}). The simplex-wise computation of Steenrod squares follows Algorithm~\ref{alg:steenrod}.

\begin{remark}\label{gauss}
Whether two \( p \)-dimensional cocycles \( c_1 \) and \( c_2 \) are cohomologous can be decided by linear algebra over \( \mathbb{Z}/2 \). Let \( A \) be the matrix of the coboundary operator
\[
\delta : C^{p-1}(X;\mathbb{Z}/2) \to C^p(X;\mathbb{Z}/2).
\]
Then \( c_1 \) and \( c_2 \) are cohomologous if and only if \( c_1-c_2 \in \operatorname{Im}(A) \), that is, the system \( A\mathbf{x} = c_1 - c_2 \) with \( \mathbf{x} \in (\mathbb{Z}/2)^{|X_{p-1}|} \) is solvable, which is decided by row-reducing the augmented matrix \( (A \mid c_1 - c_2) \) over \( \mathbb{Z}/2 \). See Section~3.3 of \citet{contessoto2021persistent}.
\end{remark}

\begin{algorithm}
\caption{Cup product computation}
\label{alg:cup_product}
\begin{algorithmic}[1]
\Require Two cochains $\sigma_1$ and $\sigma_2$, and the simplicial complex $X$.
\Ensure The cup product $\sigma = \sigma_1 \smile \sigma_2$ at cochain level.
\State $\sigma \gets [\ ]$
\If{$\dim(\sigma_1) + \dim(\sigma_2) \leq \dim(X)$}
    \For{each simplex $a=[a_0,\ldots,a_p]$ in the support of $\sigma_1$}
        \For{each simplex $b=[b_0,\ldots,b_q]$ in the support of $\sigma_2$}
            \If{$a_p=b_0$}
                \State $c \gets [a_0,\ldots,a_p,b_1,\ldots,b_q]$
                \If{$c \in X_{p+q}$}
                    \State $\sigma\gets\sigma\mathbin{\triangle}\{c\}$. \Comment{Addition over $\mathbb{Z}/2$}
                \EndIf
            \EndIf
        \EndFor
    \EndFor
\EndIf
\State \Return $\sigma$
\end{algorithmic}
\end{algorithm}

\begin{algorithm}
\caption{Cup-\(i\) coproduct (cf. Definition~\ref{cupicoprod})}
\label{alg:cupico}
\begin{algorithmic}[1]
\Require An $n$-dimensional simplex \( \sigma = [v_0, v_1, \dots, v_n] \) and an integer \(0\leq i\leq n\).
\Ensure A list of pairs \( (\sigma_L, \sigma_R) \) representing the summands of $\Delta_i(\sigma)$.
\State $\Delta_i(\sigma) \gets [\ ]$
\For{each \(U=\{u_1<\cdots<u_{n-i}\}\subseteq\{0,1,\dots,n\}\)}
    \State Initialize \( \sigma_L \leftarrow \sigma \), \( \sigma_R \leftarrow \sigma \).
    \For{$j=1,\dots,n-i$}
        \If{\(u_j+j\equiv 0\pmod 2\)}
            \State Remove \(v_{u_j}\) from \(\sigma_L\).
        \Else \Comment{$u_j+j\equiv 1\pmod 2$}
            \State Remove \(v_{u_j}\) from \(\sigma_R\).
        \EndIf
    \EndFor
    \State Add the pair \( (\sigma_L, \sigma_R) \) to $\Delta_i(\sigma)$.
\EndFor
\State \Return $\Delta_i(\sigma)$.
\end{algorithmic}
\end{algorithm}

\begin{algorithm}
    \caption{Steenrod Square computation (cf. Theorem~\ref{newsq})}
    \label{alg:steenrod}
    \begin{algorithmic}[1]
        \Require A simplicial complex $S$ and the support $c=\{\tau_j\}\subseteq S_n$ of an $n$-cocycle over $\mathbb{Z}/2$.
        \Require A natural number $k$.
        \Ensure The support of the cochain $\operatorname{Sq}^k(c)=(c\otimes c)\Delta_{n-k}$.
        \If{$k>n$}
            \State \Return $[\ ]$.
        \ElsIf{$k=0$}
            \State \Return $c$. \Comment{$\operatorname{Sq}^0$ is the identity}
        \EndIf
        \State $Q \gets [\ ]$.
        \For{each $\sigma\in S_{n+k}$}
            \State $q\gets 0\in\mathbb{Z}/2$.
            \For{each $(\sigma_L,\sigma_R)\in\Delta_{n-k}(\sigma)$ from Algorithm~\ref{alg:cupico}}
                \If{$\dim(\sigma_L)=\dim(\sigma_R)=n$ and $\sigma_L\in c$ and $\sigma_R\in c$}
                    \State $q\gets q+1\pmod 2$.
                \EndIf
            \EndFor
            \If{$q=1$}
                \State Append $\sigma$ to $Q$.
            \EndIf
        \EndFor
        \State \Return $Q$.
    \end{algorithmic}
\end{algorithm}

\subsection{Time complexity}

\begin{theorem}[\citealt{contessoto2021persistent}, Remark 4]\label{time:cupprod}
    Let $X$ be a simplicial complex, and let $c_1, c_2 \in H^*(X;\mathbb{Z}/2)$ be two cohomology classes. Denote the number of simplices in $X$ by $N$. Then, computing $c_1 \smile c_2$ via Algorithm~\ref{alg:cup_product} has complexity $O(N^2).$
\end{theorem}

\begin{lemma}\label{time:steenrod}
    Let $X$ be a simplicial complex with $N$ simplices and let $c \in H^p(X;\mathbb{Z}/2)$. Computing $\operatorname{Sq}^k(c)$ via Algorithm~\ref{alg:steenrod} takes $O(N^3)$ time.

    \begin{proof}
        There are at most $N$ simplices in $X_{p+k}$. For each such simplex
        $\sigma$, Algorithm~\ref{alg:cupico} enumerates subsets of its vertex
        set. The number of these subsets is at most $N$, since all faces of
        $\sigma$ are simplices of $X$. Thus at most $O(N^2)$ tensor pairs are
        examined. If the support of $c$ is stored as a list, each of the two
        membership tests in Algorithm~\ref{alg:steenrod} costs $O(N)$, giving
        the upper bound $O(N^3)$. (With a hash-set representation of the
        support, this part can instead be carried out in $O(N^2)$ expected
        time.)
    \end{proof}
\end{lemma}

Let $P=\{p_i\}_{i=1}^N$ be a point cloud in general position in $\mathbb{R}^D$, and let $\mathbb{X}$ be the alpha filtration obtained from $P$ on filtration interval $[s,t]$. Suppose that for $0 \leq k \leq n$, the persistent Wu classes $\{\mathbb{v}_{(k,n)}\}_{k=0}^n$ and the persistent Stiefel–Whitney classes $\{\mathbb{w}_{(k,n)}\}_{k=1}^n$ of dimension $n$ for $\mathbb{X}$ exist.
Assume in addition that, for each $0\leq k\leq n$, the $k$-th Wu class of dimension $n$ of $X^t$ exists and is unique.

\begin{theorem}
    With the notation above, let $S$ denote the number of simplices in $\mathbb{X}$. Then $S=O(N^{\lceil D/2 \rceil})$. Computing the persistent Wu and Stiefel–Whitney classes of dimension $n$ on $[s,t]$ via Algorithm~\ref{alg:pers_sw} takes $O(nS^4)=O(nN^{4\lceil D/2 \rceil})$ time, not counting the construction of $\mathbb{X}$.

    \begin{proof}
    The Delaunay complex of \( P \subset \mathbb{R}^D \) has \( O(N^{\lceil D/2 \rceil}) \) simplices \citep{seidel1995upper}, and so does \( \mathbb{X} \). Its persistent cohomology together with representative cocycles is computed in \( O(S^3) \) time (\citealt{edelsbrunner2002topological}, Section~4). Fix \( 1\le k\le n \). The bases $\{\alpha_i\}$, $\{x_j\}$, $\{\gamma_l\}$ of Proposition~\ref{prop:wusolve} have at most $S$ elements each, so the system $L\xi=q$ has $O(S)$ unknowns and $O(S^2)$ equations. Assembling it uses $O(S^2)$ cup products, each of cost $O(S^2)$ (Theorem~\ref{time:cupprod}), and $O(S)$ Steenrod squares, each of cost $O(S^3)$ (Lemma~\ref{time:steenrod}). Expressing their results in the basis $\{\gamma_l\}$ uses a single row reduction of the coboundary matrix, of cost $O(S^3)$ (Remark~\ref{gauss}). Solving $L\xi=q$ by Gaussian elimination costs $O(S^4)$ (\citealt{golub2013matrix}, Chapter~3), so each Wu class is obtained in $O(S^4)$ time, and applying the Wu formula (Definition~\ref{persswtype}) adds $O(S^3)$. Summing over $k$, and adding the verifications of the Wu criterion at the bar endpoints in Algorithm~\ref{alg:pers_sw} (Remark~\ref{endpoint}), the total is $O(nS^4)$.
    \end{proof}
\end{theorem}

\section{Computational examples}\label{sec:examples}

For our computational experiments, we used the GUDHI library \citep{maria2014gudhi} to construct the alpha filtration and the Dionysus 2 library \citep{dionysus2} to compute persistent cohomology and extract representative cocycles. For the point clouds below, the filtration parameter is the square of the radius. All implementations are publicly available at \citet{Gang2025}.

\subsection{\texorpdfstring{$\mathbb{CP}^2$ blown up at 1 point and $S^2 \times S^2$}{CP2 blown up at 1 point and S2 x S2}}

Here $\overline{\mathbb{CP}}^2$ denotes $\mathbb{CP}^2$ with the opposite orientation. The blow-up of \( \mathbb{CP}^2 \) at a point is diffeomorphic to \( \mathbb{CP}^2 \# \overline{\mathbb{CP}}^2 \) (\citealt{huybrechts2005complex}, Proposition 2.5.8). The two manifolds \( \mathbb{CP}^2 \# \overline{\mathbb{CP}}^2 \) and \( S^2 \times S^2 \) have isomorphic cohomology groups over $\mathbb{Z}/n$ for every positive integer $n$, but are distinguished by their integral intersection forms.

\begin{figure}[ht]
    \centering
    \includegraphics[width=0.8\textwidth]{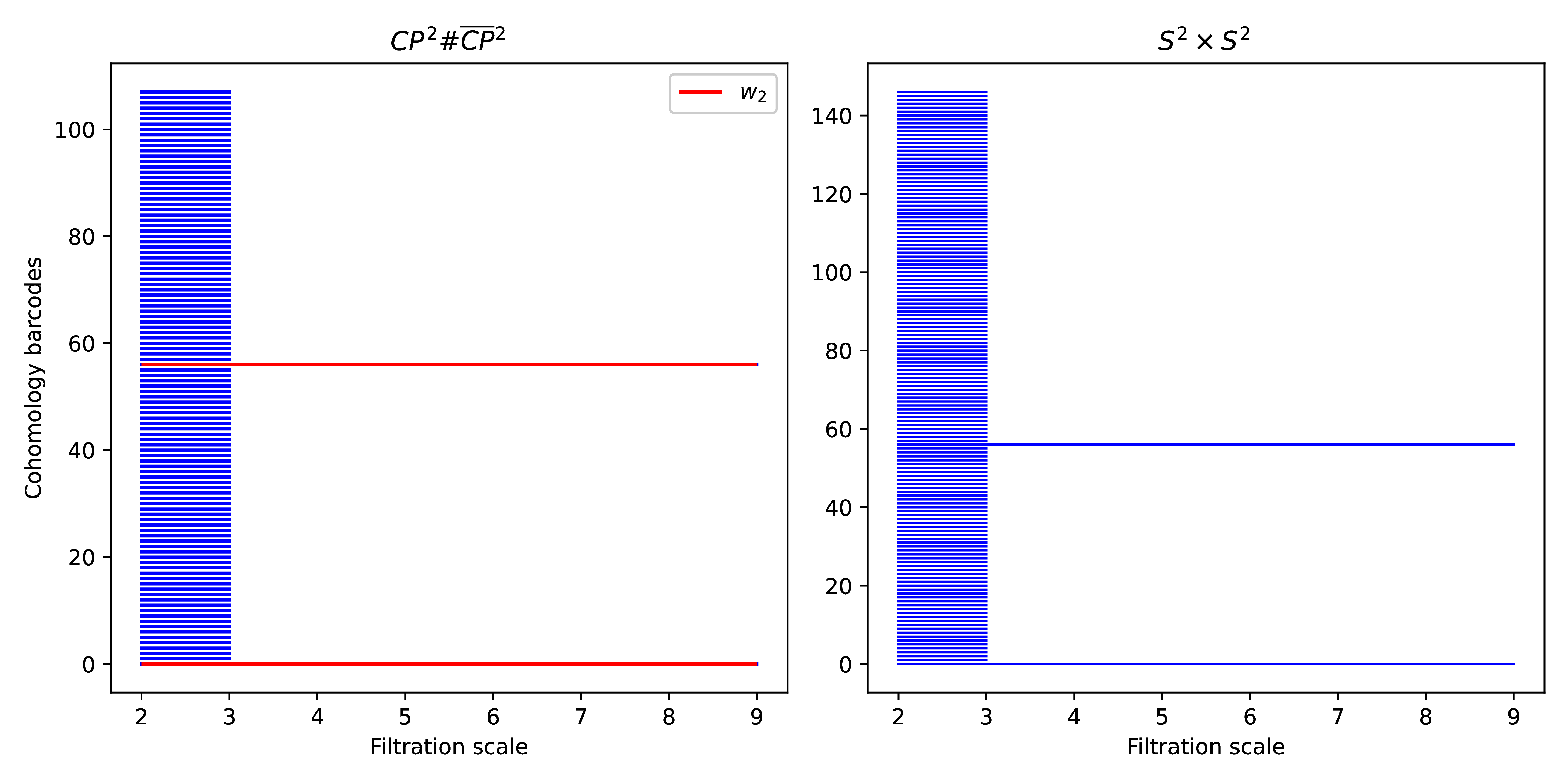}
    \caption{The second persistent cohomology barcodes and the second Stiefel–Whitney classes of dimension 4 at the maximum filtration scale computed for $\mathbb{CP}^2 \# \overline{\mathbb{CP}}^2$ and $S^2 \times S^2$. The horizontal axis represents the filtration scale, while the vertical axis indicates the persistent cohomology barcodes. Red bars correspond to cohomology classes whose sum yields the second Stiefel–Whitney class of dimension 4.}
    \label{fig:sw2_barcodes}
\end{figure}

To construct the triangulation of \( \mathbb{CP}^2 \# \overline{\mathbb{CP}}^2 \), we utilized the triangulation of \( \mathbb{CP}^2 \) from the Steenroder library \citep{lupo2022persistence}. For $S^2\times S^2$, we used the 16-vertex triangulation of \citet{bagchi2011icosahedron}. We constructed filtered simplicial complexes from the triangulations of \( \mathbb{CP}^2 \# \overline{\mathbb{CP}}^2 \) and \( S^2 \times S^2 \), assigning filtration scales to simplices based on their dimensions. Using these filtered complexes, we computed persistent cohomology over \( \mathbb{Z}/2 \) and applied our algorithm to determine the first and second Stiefel–Whitney classes of dimension $4$ for each case. Since both manifolds are orientable, the first Stiefel–Whitney class is trivial in both cases. Figure~\ref{fig:sw2_barcodes} presents the results for the second persistent cohomology and the second Stiefel–Whitney class of dimension $4$ at the maximum filtration scale. As expected, both manifolds exhibit two independent second cohomology elements. However, for \( \mathbb{CP}^2 \# \overline{\mathbb{CP}}^2 \), the sum of the cohomology classes corresponding to the two red bars yields a nontrivial second Stiefel–Whitney class, whereas for \( S^2 \times S^2 \), the second Stiefel–Whitney class remains trivial.

This example illustrates how persistent Stiefel--Whitney classes can provide additional information for distinguishing manifolds beyond ordinary persistent cohomology. These classes are computed using the cup product and Steenrod squares through the Wu formula and can detect topological differences between manifolds with isomorphic cohomology groups.

\subsection{Projective spaces of lines}
In image analysis, topological data analysis offers a framework for capturing geometric and structural features that conventional pixel-based methods often fail to detect. By representing local image patches as high-dimensional point clouds, persistent homology enables the extraction of topological features that encode shape and texture information, complementing traditional feature extraction methods~\citep{carlsson2008local, singh2023topological, hensel2021survey, clough2020topological}.

\citet{scoccola2023approximate} proposed a method to compute characteristic classes of vector bundles by estimating tangent spaces using local PCA. In their LINES dataset, each grayscale image of size \(10 \times 10\) represents a fuzzy line with a certain slope and offset and is treated as a point in \( \mathbb{R}^{100} \). They computed Stiefel–Whitney classes associated with these data points and argued that the result suggests that the dataset is parametrized by the projective plane \( \mathbb{RP}^2 \), which has nontrivial first and second Stiefel–Whitney classes.

We used $250$ images from the LINES dataset of \citet{scoccola2023approximate}. To reduce dimensionality while preserving essential structure, we applied principal component analysis (PCA) \citep{abdi2010principal}, projecting the dataset into $\mathbb{R}^5$, and rescaled the point cloud by its largest norm.

After PCA and rescaling, we sampled $120$ of these points uniformly with replacement. To each selected point we added a vector whose coordinates were independent Gaussian random variables with mean $0$ and standard deviation $0.08$. We appended these perturbed copies to the original sample, obtaining $370$ points in $\mathbb{R}^5$.

\begin{figure}[ht]
    \centering
    \includegraphics[width=\textwidth]{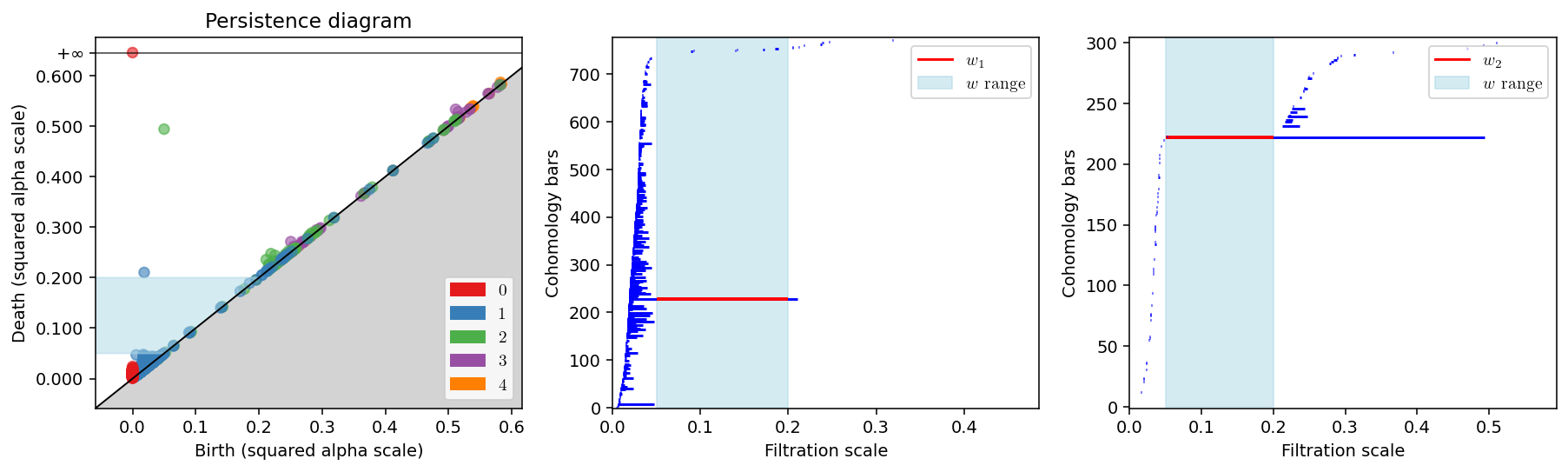}
    \caption{Persistence diagram (left) and barcodes in degrees $1$ (middle) and $2$ (right), computed over $\mathbb{Z}/2$ from the full alpha filtration of the LINES data with $120$ added points. Red segments show the first and second persistent Stiefel–Whitney classes of dimension $2$ on the shaded interval $[s,0.2]$, where $s\approx 0.0499$ is the birth scale of the bar representing the fundamental class.}
    \label{fig:sw_persistence}
\end{figure}

From this enlarged point cloud, we constructed an alpha filtration and applied our algorithm on $[s,t]$, with $s\approx 0.0499$ and $t=0.2$. In addition to the bars that survive throughout this interval, $11$ bars in degree $1$ and $8$ bars in degree $2$ end inside it. The persistent Wu criterion holds throughout $[s,t]$, and both computed Stiefel–Whitney classes are nonzero, in agreement with $w_1(\mathbb{RP}^2)$ and $w_2(\mathbb{RP}^2)$. Figure~\ref{fig:sw_persistence} shows the persistence diagram and barcodes of the full filtration, with the computed classes marked on $[s,t]$.

We checked the homotopy types at all $7{,}225$ changes of the filtration on $[s,t]$ by applying elementary simplicial collapses, contracting a spanning tree of the $1$-skeleton, and cancelling pairs of $1$- and $2$-cells. For every $r\in[s,t]$, we verified $X^r\simeq\mathbb{RP}^2\vee Y^r$, where $Y^r$ is a finite wedge sum of $1$- and $2$-spheres, possibly a point. We also checked that $X^t\simeq\mathbb{RP}^2$. Although these computations do not establish the compatibility of the retractions required in Theorem~\ref{thm:wedge}, they provide empirical support for its hypotheses and for the model of noise by wedges of spheres discussed in Remark~\ref{rem:noise}.

\subsection{Conformation space of cyclooctane}
Cyclooctane (\(\text{C}_8\text{H}_{16}\)) is a flexible cyclic alkane with a complex conformation space due to its increased degrees of freedom. Topological data analysis has been applied to this space using persistent homology to analyze
its geometric and topological structure. With bond lengths and angles fixed, the conformation space of cyclooctane is nonmanifold, described as the union of a Klein bottle and \( S^2 \), which intersect along two disjoint circles \citep{martin2010topology, lupo2022persistence, eliel1994stereochemistry, membrillo2019topology, adams2011javaplex}. Figure~\ref{fig:cyclooctane_conformation} illustrates the molecular structure of cyclooctane and its conformation space projected into \( \mathbb{R}^3 \).

\begin{figure}[ht]
    \centering
    \includegraphics[width=0.8\textwidth]{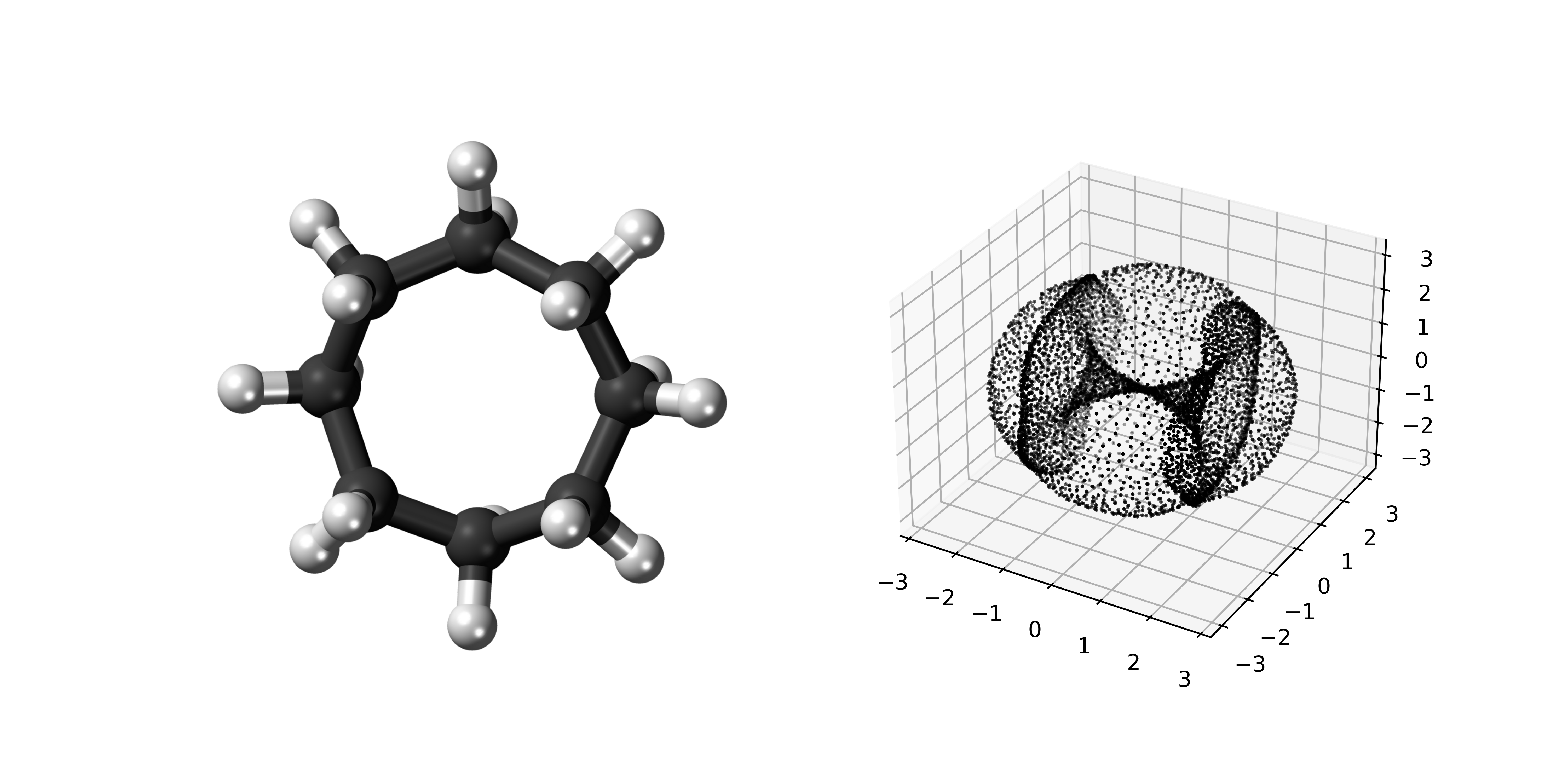}
    \caption{Molecular structure of cyclooctane (\(\text{C}_8\text{H}_{16}\)) and a point cloud sampled from its conformation space projected into \(\mathbb{R}^3\). With bond lengths and angles fixed, the conformation space consists of a Klein bottle and a 2-sphere intersecting along two disjoint circles.}
    \label{fig:cyclooctane_conformation}
\end{figure}

Following \citet{lupo2022persistence}, we started with 6040 points in \(\mathbb{R}^{24}\) sampled from the conformation space of cyclooctane, originally introduced by \citet{martin2010topology}. \citet{stolz2020geometric} applied a method based on local persistent cohomology to identify 627 candidate singular points. \citet{lupo2022persistence} removed these points and used the HDBSCAN clustering algorithm \citep{campello2013density} to segment the remaining points into four clusters, presumably corresponding to an open subset of the Klein bottle and three subsets of \( S^2 \). Following these processes, we obtained a set of points from the presumed Klein bottle component of the conformation space.

We selected $500$ points from the Klein bottle component by farthest-point sampling and applied the UMAP algorithm \citep{mcinnes2018umap} to embed them into $\mathbb{R}^6$. After this embedding, we sampled $160$ of these points uniformly with replacement and added independent Gaussian noise with mean $0$ and standard deviation $0.12$ in each coordinate. Appending the perturbed copies to the original sample produced $660$ points in $\mathbb{R}^6$.

\begin{figure}[ht]
    \centering
    \includegraphics[width=\textwidth]{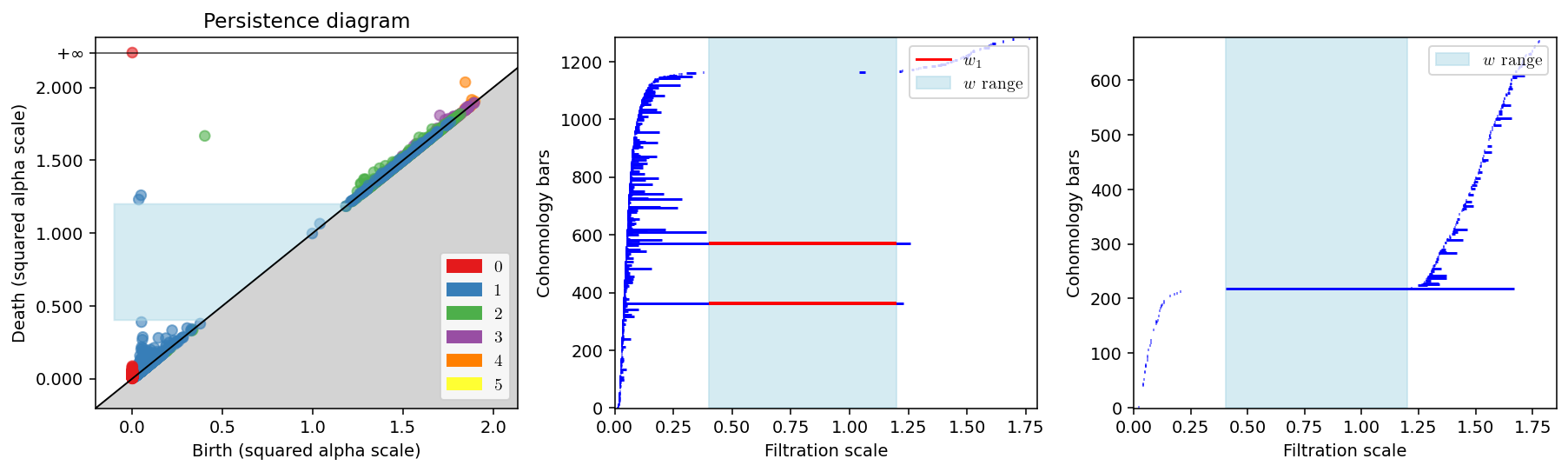}
    \caption{Persistence diagram (left) and barcodes in degrees $1$ (middle) and $2$ (right), computed over $\mathbb{Z}/2$ from the full alpha filtration of the cyclooctane data with $160$ added points. The shaded interval is $[s,1.2]$, where $s\approx 0.4018$ is the birth scale of the bar representing the fundamental class. The two red segments in the middle panel mark classes whose sum is the first persistent Stiefel–Whitney class of dimension $2$. The second persistent Stiefel–Whitney class is zero.}
    \label{fig:sw_klein}
\end{figure}

We computed the persistent Stiefel–Whitney classes of dimension $2$ over $[s,t]$, with $s\approx 0.4018$ and $t=1.2$. Figure~\ref{fig:sw_klein} shows the results. The first class is the sum of the two classes marked in red, and the second class is zero, in agreement with the known Stiefel–Whitney classes of the Klein bottle. There are $3$ additional bars in degree $1$ and $1$ in degree $2$ that end inside $[s,t]$. The persistent Wu criterion holds throughout this interval despite these additional classes.

Using the same procedure, we checked the homotopy types at all $25{,}563$ changes of the filtration on $[s,t]$. Let $K$ denote the Klein bottle. For every $r\in[s,t]$, we verified $X^r\simeq K\vee Y^r$, where $Y^r$ is a finite wedge sum of $1$- and $2$-spheres, possibly a point. At the final scale, $X^t\simeq K$. In both experiments, the additional summands $Y^r$ satisfy the vanishing condition on the positive Steenrod squares in Theorem~\ref{thm:wedge}. These experiments show that our algorithm can compute the expected persistent Stiefel--Whitney classes from datasets arising in applications even when added noise produces additional bars.

\section{Conclusion}

We introduced persistent Wu classes and persistent Stiefel--Whitney classes of dimension $n$ using the cup product and Steenrod squares on persistent cohomology. The Wu criterion reduces to a linear system over $\mathbb{Z}/2$, leading to an algorithm that runs in polynomial time in the number of simplices without estimating tangent spaces or constructing a classifying map. Under the hypotheses of Proposition~\ref{swequal} or Theorem~\ref{thm:wedge}, the persistent Stiefel--Whitney classes represent those of the tangent bundle of a closed smooth manifold $M$. Theorem~\ref{thm:wedge} allows additional wedge summands on whose reduced cohomology all positive Steenrod squares vanish, provided that the retractions onto $M$ are compatible with the filtration inclusions and the final space is homotopy equivalent to $M$.

The experiments illustrate how these classes can distinguish manifolds with isomorphic cohomology groups and recover the expected persistent classes from datasets arising in applications even when added noise produces additional bars. In both noisy examples, we verified that the complexes are homotopy equivalent to the underlying manifold wedged with $1$- and $2$-spheres throughout the chosen intervals. These observations support the model of noise discussed in Remark~\ref{rem:noise}.

Further work is needed to identify geometric and sampling conditions that ensure the wedge decompositions and compatible retractions in Theorem~\ref{thm:wedge} for noisy point clouds. Another question is how our classes relate to those computed from approximate vector bundles \citep{scoccola2023approximate}. Results on recovering the homotopy type of a manifold with Vietoris--Rips complexes \citep{latschev2001vietoris, kim2020homotopy} motivate developing efficient computations on Rips filtrations. An important computational problem is how to reduce the number of simplices while preserving the cup product and Steenrod squares needed for the computation. Finally, it remains to determine whether the approach can be extended to Pontryagin, Euler, or Chern classes, and what additional bundle data would be needed.

\newpage
\bibliography{sn-bibliography}
\end{document}